\documentclass[12pt]{article}
\usepackage{amsmath, amssymb, amsfonts, amsthm}
\usepackage{graphicx}
\usepackage{color}
\usepackage[margin=1in]{geometry}
\usepackage[all]{xy}
\newdir{d}{{\cong}}
\usepackage{ stmaryrd }
\usepackage{ textcomp }
\usepackage{enumerate}
\usepackage{tikz}
\theoremstyle{plain}
\newtheorem*{thma}{Theorem A}
\newtheorem*{thmb}{Theorem B}
\newtheorem*{thmc}{Theorem C}
\newtheorem*{corollarya}{Corollary}
\newtheorem{thm}{Theorem}
\newtheorem{lemma}[thm]{Lemma}
\newtheorem{corollary}[thm]{Corollary}
\newtheorem{proposition}[thm]{Proposition}
\theoremstyle{definition}
\newtheorem{definition}[thm]{Definition}
\newtheorem{remark}[thm]{Remark}
\theoremstyle{example}

\theoremstyle{plain}

\makeatletter
\newcommand\appendix@section[1]{%
\refstepcounter{section}%
\orig@section*{Appendix \@Alph\c@section: #1}%
\addcontentsline{toc}{section}{Appendix \@Alph\c@section: #1}%
}
\let\orig@section\section
\g@addto@macro\appendix{\let\section\appendix@section}
\makeatother
\title{A Bott periodicity proof for real graded $C^\ast$-algebras}
\author{Sarah L. Browne}

\begin{document}
\maketitle

\begin{abstract}
We give a proof of Bott periodicity for real graded $C^\ast$-algebras in terms of $K$-theory and $E$-theory. Guentner and Higson~\cite{GH04} proved a similar result in the complex graded case but we extend this to cover all graded $C^\ast$-algebras. We obtain the $8$-fold periodicity in $E$-theory by constructing two maps that are inverse to each other. 
\end{abstract} 

\section*{Introduction}
$K$-theory has been studied since its description by Atiyah~\cite{Ati67} and Hirzebruch. Its generalisation from topological spaces to $C^\ast$-algebras has made it an essential part of understanding $C^\ast$-algebras. It became a key tool in the study of $C^\ast$-algebras due to the work by George Elliott on classification of $AF$-algebras~\cite{Ell76} and has been used ever since. Additionally there are uses in physics, geometry and algebraic topology as detailed by Connes in~\cite{Con94}.

$E$-theory was formed by Higson~\cite{Hig90} as a categorical object since $KK$-theory did not have the excision property given to $E$-theory. It was made into a concrete definition in terms of asymptotic morphisms by Connes and Higson~\cite{CH90} and further work has been done by Guentner, Higson and Trout~\cite{GHT00}, Dardarlat-Meyer~\cite{DM12} among others. 

In this paper we extend the proof in~\cite{GH04} of Bott periodicity of $E$-theory to real graded $C^\ast$-algebras. Bott periodicity is a key feature of the $K$-theory and $E$-theory construction and has many formulations and proofs. We give details of ours, in which we use Clifford algebras and operators. 

We construct the \emph{Bott map}, $\beta \colon \mathcal{S} \rightarrow C_0(V,\text{Cliff}(V))$ in the real case using functional calculus as in~\cite{GH04} and we prove our main result: 

\begin{thma}
For a finite dimensional Euclidean vector space $V$ over $\mathbb{R}$, the Bott map induces the isomorphism of $K$-theory 
\[\beta_{\ast} \colon K(\mathbb{R}) \rightarrow K(\mathbb{R} \widehat{\otimes} C_0(V,\text{Cliff}(V))).\]
\end{thma}

In order to prove this, we construct an inverse map $\alpha_t$ to $\beta$, $$\alpha_t \colon \mathcal{S} \widehat{\otimes} C_0(V, \text{Cliff}(V)) \rightarrow \mathcal{K}(\mathcal{H}(V))$$ and we will show that the induced composition $\alpha_{\ast}\beta_{\ast}$ is equivalent to a simpler map $\gamma_t \colon \mathcal{S} \rightarrow \mathcal{K}(\mathcal{H}(V))$. The map $\gamma_t$ will give a $K$-theory equivalence.
Additionally we use the Atiyah rotation trick~\cite{Ati68} to prove $ \beta_{\ast}\alpha_{\ast}$ yields the identity in $K$-theory. 
We will obtain the following Corollary and Theorem by showing that when $V=\mathbb{R}$ then $C_0(V,\text{Cliff}(V)) = \Sigma \mathbb{R} \widehat{\otimes} \mathbb{R}_{1,0}:$

\begin{corollarya}
We have an isomorphism
\[K(\mathbb{R}) \cong K(\Sigma \mathbb{R} \widehat{\otimes} \mathbb{R}_{1,0}).\]
\end{corollarya}

\begin{thmb}
For a real graded $C^\ast$-algebra $A$, we have a natural isomorphism
\[K(A) \cong K(\Sigma A \widehat{\otimes} \mathbb{R}_{1,0}).\]
\end{thmb}

Finally we will transfer this into $E$-theory by using the relation between $K$- and $E$-theory, namely $$K(A) := E(\mathbb{R},A),$$
and obtain
\begin{thmc}
For real graded $C^\ast$-algebras $A$ and $B$, 
\[E(A,B) \cong E( A , \Sigma B \widehat{\otimes} \mathbb{R}_{1,0}).\]
\end{thmc}

This result will be obtained by the tensor product bifunctor property of $E$-theory and the facts that both $\beta$ and $\alpha_t$ give classes in $E$-theory which are invertible. Then we use properties of Clifford algebras to formulate the $8$-fold periodicity in $E$-theory. 

\section*{Acknowledgements}
The author would like to thank her PhD supervisor, Paul Mitchener, for the support and conversations that allowed this paper to be written. 
Additionally, the author thanks the EPSRC for her funding for her PhD. 
Then the author also thanks the following for conversations, email correspondence etc: Erik Guentner, Nigel Higson, Jamie Gabe, Ulrich Pennig and many others. 

\section{$K$- and $E$-theory}
In this section we recall the definitions of $K$- and $E$-theory alongside the fundamental definitions of aspects of real graded $C^\ast$-algebras. 
We begin by recalling the definitions of grading on a $C^\ast$-algebra. Let $A$ be a graded $C^\ast$-algebra and let $\text{deg}(a)$ denote the degree of the element $a \in A$.
\begin{remark}\label{grading on compacts}
Let $\mathcal{H}$ be a Hilbert space equipped with an orthogonal decomposition 
\[\mathcal{H} = \mathcal{H}_0 \oplus \mathcal{H}_1.\]
Then the $C^\ast$-algebra $\mathcal{K}(\mathcal{H})$ of compact operators on $\mathcal{H}$ is graded. We say $T: \mathcal{H} \rightarrow \mathcal{H}$ is \emph{even} if $T[\mathcal{H}_0] \subseteq \mathcal{H}_0$ and $T[\mathcal{H}_1] \subseteq \mathcal{H}_1$, and \emph{odd} if $T[\mathcal{H}_0] \subseteq \mathcal{H}_1$ and $T[\mathcal{H}_1] \subseteq \mathcal{H}_0$.
\end{remark}
One can also define the graded tensor product $A \widehat{\otimes} B$:-

\begin{definition}\label{graded tensor product}
Let $A$ and $B$ be graded $C^\ast$-algebras with gradings $\delta_A$ and $\delta_B$ respectively. Then  define $A \widehat{\otimes} B$ to be the completion of the algebraic tensor product of $A$ and $B$ in the norm 
\[|| \sum_i a_i \otimes b_i || = \sup || \sum_i \varphi(a_i) \psi(b_i)||\]  
where $\varphi \colon A \rightarrow C$, $\psi \colon A \rightarrow C$ are graded $\ast$-homomorphisms to a common graded $C^\ast$-algebra $C$. 
We equip $A \widehat{\otimes} B$ with involution, multiplication and grading defined by:
\begin{enumerate}
\item $(a \widehat{\otimes} b)^{\ast} = (-1)^{\text{deg}(a) \text{deg}(b)} a^{\ast} \otimes~  b^{\ast}$
\item $(a \widehat{\otimes} b)(c \widehat{\otimes} d) =  (-1)^{\text{deg}(b) \text{deg}(c)}(ac \otimes bd)$
\item $\gamma(a \widehat{\otimes} b) = \alpha(a) \otimes \beta(b)$
\end{enumerate}
Extending by linearity gives $A \widehat{\otimes} B$. 
\end{definition}
Due to Palmer~\cite{Pal84}, we have the following definition:
\begin{definition}
A \emph{real $C^\ast$-algebra} $A$ is a real Banach algebra with involution satisfying the $C^\ast$-identity and additionally the property that $1+x^\ast x$ is invertible for all $x \in A$. 
\end{definition}

Now we are almost ready to state the definition of the real $K$-theory groups. 
Let $C([0,1],B)$ denote the set of continuous functions from $[0,1]$ to $B$. Then recall that a \emph{homotopy} between two graded $\ast$-homomorphisms $\varphi, \psi \colon A \rightarrow B$ is a graded $\ast$-homomorphism $\theta \colon A \rightarrow C([0,1],B)$ such that 
\[\theta(a)(0) = \varphi(a) ~ \text{and} ~ \theta(a)(1) = \psi(a).\]
Denote the set of homotopy classes of graded $\ast$-homomorphisms from $A$ to $B$ by $[A,B]$. 

Consider the algebra of continuous real valued functions that vanish at infinity, $C_0(\mathbb{R})$, then we can equip this with the grading that even functions are even and odd functions are odd. Denote this algebra by $\mathcal{S}$.
\begin{definition}
For a real graded $C^\ast$-algebra $A$, the $K$-theory groups are given by 
\[K_n(A) = [\mathcal{S}, \Sigma^n A \widehat{\otimes} \mathcal{K}(\mathcal{H})],\]
for $n \geq 0$, where $\mathcal{K}(\mathcal{H})$ denotes the compact operators on some real graded Hilbert space with grading as in Remark~\ref{grading on compacts}. Here $\Sigma A = C_0(\mathbb{R}) \widehat{\otimes} A$ where $C_0(\mathbb{R})$ has the trivial grading, and $\Sigma^n A$ is defined iteratively. 
\end{definition}
The above definition is taken from~\cite{GH04}.

Now to formulate the $E$-theory groups we need the notion of a graded asymptotic morphisms and  homotopy of these. 

\begin{definition}
Let $A$ and $B$ be real graded $C^\ast$-algebras with gradings $\delta_A$ and $\delta_B$ respectively. A \emph{graded asymptotic morphism} $\varphi \colon A \dashrightarrow B$ is a family of functions $\{\varphi_t\}_{t \in [1,\infty)} \colon A \dashrightarrow B$ such that: 
\begin{enumerate}
\item the map $t \mapsto \varphi_t(a)$, from $[1,\infty)$ to  $B$ is continuous and bounded for each $a \in A$, 
\item $\lim_{t \to \infty} || \varphi_t (ab) - \varphi_t (a)\varphi_t(b) || = 0,$ for each $a,b \in A$, 
\item $\lim_{t \to \infty} || \varphi_t (a + \lambda b) - \varphi_t (a) - \lambda \varphi_t (b) || = 0,$ for each $a,b \in A$, $\lambda \in \mathbb{R}$, 
\item $\lim_{t \to \infty} || \varphi_t (a^*) - \varphi_t (a)^* || = 0,$ for each $a \in A$, 
\item $\lim_{t \to \infty} || \varphi_t(\delta_A(a)) - \delta_B(\varphi_t(a)) || = 0$ for each $a \in A$. 
\end{enumerate}
\end{definition}

Note that we can define $A \widehat{\otimes} B$ in terms of graded asymptotic morphisms by replacing the graded $\ast$-homomorphisms $\varphi, \psi$ in Definition~\ref{graded tensor product} by graded asymptotic morphisms. 

Before we state what a homotopy of graded asymptotic morphisms is, we first consider the composition of a graded asymptotic morphisms and a $\ast$-homomorphism. 
It will become reasonable to just consider the following two morphisms: 
\[\alpha_t \colon\mathcal{S} \widehat{\otimes} A \rightarrow B ~ \text{and} ~ \beta \colon \mathcal{S} \widehat{\otimes} B \dashrightarrow C.\]
For this we need to construct a $\ast$-homomorphism 
\[\Delta \colon \mathcal{S} \rightarrow \mathcal{S} \widehat{\otimes} \mathcal{S},\]
as in~\cite{GH04}. 
First notice that the functions 
\[u(x)= e^{-x^2} ~\text{and} ~ v(x) = x e^{-x^2}\]
are functions contained in $\mathcal{S}$ and $u$ is even and $v$ is odd. 

Now set $A$ to be the algebra generated by $u$ and $v$. Then by Theorem~B in~\cite{Sim63} we have the following result: 

\begin{lemma}\label{Agenerates}
The algebra $A$ is dense in $\mathcal{S}$. 
\end{lemma}
Hence $u$ and $v$ generate $\mathcal{S}$. 

Let $\mathcal{S}_R$ be the set of continuous functions on the interval $[-R,R]$; we have a surjection $\pi \colon \mathcal{S} \rightarrow \mathcal{S}_R$ defined by restriction. Let $X_R \in \mathcal{S}_R$ be the function $x \mapsto x$, then note that $X_R$ is odd. Now if $f \in \mathcal{S}$, by functional calculus we have an element 
\[f(X_R \widehat{\otimes} 1 + 1 \widehat{\otimes} X_R) \in \mathcal{S}_R \widehat{\otimes} \mathcal{S}_R,\]
where $1$ denotes the function $1$. Then by~\cite{GH04} we have a graded $\ast$-homomorphism $\Delta \colon \mathcal{S} \rightarrow \mathcal{S} \widehat{\otimes} \mathcal{S}$ such that \[\Delta(u) = u \widehat{\otimes} u ~~~ \textit{and}~~~ \Delta(v) = u \widehat{\otimes} v +  v \widehat{\otimes} u.\]

Then we can define the composition of $\ast$-homomorphism $\beta$ and an asymptotic morphism $\alpha_t$ by:- 

\[\xymatrixcolsep{3pc}\xymatrix{\mathcal{S} \widehat{\otimes} A \ar[r]^{\Delta \widehat{\otimes} \text{id}_{A}} & \mathcal{S} \widehat{\otimes} \mathcal{S} \widehat{\otimes} A \ar@{.>}[r]^{\text{id}_{\mathcal{S}} \widehat{\otimes} \alpha_t}& \mathcal{S} \widehat{\otimes} B \ar[r]^{\beta}& C .} \]
 We define the composition of he form $\beta \circ \alpha_t$ similarly. 

\begin{definition}
A \emph{homotopy} of graded asymptotic morphisms $\varphi_t, \psi_t \colon A \dashrightarrow B$ is a graded asymptotic homomorphism $\theta_t \colon A \dashrightarrow C([0,1],B)$ such that 
\[\theta_t(a)(0) = \varphi_t(a) ~ \text{and} ~ \theta_t(a)(1) = \psi_t(a).\]
\end{definition}
Denote the set of homotopy classes of graded asymptotic morphisms from $A$ to $B$ by $\llbracket A,B \rrbracket$. 
Then the definition, from~\cite{GH04}, of the $E$-theory groups is given by: 
\[E^n(A,B) = \llbracket \mathcal{S} \widehat{\otimes} A \widehat{\otimes} \mathcal{K}(\mathcal{H}), \Sigma^n B  \widehat{\otimes} \mathcal{K}(\mathcal{H}) \rrbracket, \]
where $\Sigma^n B$ for $n \geq 0$ is the $n$-fold suspension of a $C^\ast$-algebra and its grading comes from $B$. 
For information on why these form groups see Lemma~2.1 in~\cite{GH04} for a proof of additive inverses. There is also a full description of this in the author's PhD thesis. 

Additionally we note that there is an $E$-theory product 
\[E(A,B) \times  E(B,C) \rightarrow E(A,B),\]
which is defined composition at the level of homotopy of asymptotic morphisms. We denote this by $x \times y$ where $x \in E(A,B)$ and $y \in E(B,C)$. 

\section{Bott periodicity}
Here we extend the Bott map of Guentner and Higson~\cite{GH04}. Just as in that article we start by thinking about the Bott periodicity in $K$-theory and then extend it to $E$-theory. 

Let $\text{Cliff}(V)$ denote the \emph{Clifford algebra} of a finite dimensional Euclidean real vector space $V$. Then define a function $P \colon V \rightarrow \text{Cliff}(V)$ by $P(v) = v$. Notice that $P(v)^2 = ||v||^2 $. 
Now we require a function in $C_0(V, \text{Cliff}(V))$  so using functional calculus, take $f \in \mathcal{S}$ then we define the map $f(P)  \colon V \rightarrow \text{Cliff}(V)$ by: 
\[v \mapsto f(P(v)).\]
This gives a $\ast$-homomorphism from $\mathcal{S}$ to $C_0(V, \text{Cliff}(V))$, and hence the following definition:- 
\begin{definition}
We define the \emph{Bott element} $b$ to be the class in the group $K(C_0(V, \text{Cliff}(V)))$ induced from the $\ast$-homomorphism $\beta \colon \mathcal{S} \rightarrow C_0(V, \text{Cliff}(V))$ defined by $f \mapsto f(P)$.
\end{definition}
\begin{thm}\label{Bott periodicty statement in beta}
For a finite dimensional Euclidean vector space $V$ over $\mathbb{R}$, the Bott map induces the isomorphism of $K$-theory 
\[\beta_{\ast} \colon K(\mathbb{R}) \rightarrow K(\mathbb{R} \widehat{\otimes} C_0(V,\text{Cliff}(V))),\]
defined by the formula $\beta_{\ast}(x) = x \times b$, where $x \in K(\mathbb{R})$. 
\end{thm}

To prove this as stated in the introduction, we need a collection of ingredients described in the next few pages.

Denote by $\mathcal{H}(V)$ the real Hilbert space $L^2(V, \text{Cliff}(V))$ which is the Hilbert space of square-integrable $\text{Cliff}(V)$-valued functions on $V$. Note that this is graded and the grading is the one coming from that of $\text{Cliff}(V)$. 

The following definition may seem irrelevant to our aim but it will be very useful for defining the operators we need to define the maps inducing Bott periodicity. 
\begin{definition}
Let $e, f$ be elements in $V$. 
Define linear operators $e, \hat{f} \colon \text{Cliff}(V) \rightarrow \text{Cliff}(V)$ on a finite-dimensional Hilbert space under $\text{Cliff}(V)$ by
\[e(x) = e x,\] 
\[\hat{f}(x) = (-1)^{\text{deg}(x)} x f .\]
\end{definition}
Let $\text{Sch}(V)$ denote the dense subspace of $\mathcal{H}(V)$ of Schwartz-class $\text{Cliff}(V)$-valued functions. For a definition of Schwartz-class see Section~2.1 of \cite{HRT}.

\begin{definition}
Define the Dirac operator $D \colon \text{Sch}(V) \rightarrow \mathcal{H}(V)$  by 
\[ (Df)(v) = \sum_{i=1}^{n} \hat{e_i} (\frac{\partial f}{\partial{x_i}}(v)),\]
where the $e_i$'s form an orthonormal basis of $V$ and each $x_i$ is the corresponding coordinate in $V$.  
\end{definition}

For $h \in C_0(V,\text{Cliff}(V))$ let $M_{h}$ denote the operator of pointwise multiplication by $h$ on the space $\mathcal{H}(V)$. So \[M_{h}(g)(v) = h(v) g(v).\]

\begin{lemma}\label{compactoperator}
The Dirac operator on $V$ is formally self-adjoint. If $f \in \mathcal{S}$ and $h \in C_0(V,\text{Cliff}(V))$, the product $f(D)M_h$ is a compact operator on $\mathcal{H}(V)$.
\end{lemma}
For a proof see~\cite{GH04}, Lemma 1.8, since this proof is still valid when we take real finite dimensional Euclidean vector spaces. 

It will be convenient to write $h \in C_0(V,\text{Cliff}(V))$ as $h_t$, with $h_t(v) = h(t^{-1}v)$ for $t \in [1, \infty)$. 

\begin{definition}
Define the graded commutator of elements $a,b$ in a real graded $C^\ast$-algebra $A$ by 
\[[a,b] = ab- (-1)^{\text{deg}(a) \text{deg}(b)} ba,\]
which extends by linearity to all elements in $A$. 
\end{definition}
Let $\mathbb{R}_{n,0}$ be the Clifford algebra generated by odd elements $e_1, \ldots, e_n$ such that $e_i^2=1$ and $e_i e_j = -e_j e_i$. Then by substituting $V=\mathbb{R}^n$ the following result makes sense: 
\begin{lemma}\label{doublecommutatordisappears}
For every $f \in \mathcal{S}, h \in C_0(\mathbb{R}^n,\mathbb{R}_{n,0})$, with the finite dimensional Euclidean vector space $\mathbb{R}^n$ having Dirac operator $D$, 
\[\lim_{t \to \infty} || [f(t^{-1}D), M_{h_t}] || = 0, \]
where  $M_{h_t}$ is a bounded linear operator on $\mathcal{H}(\mathbb{R}^n)$ and $f(t^{-1}D)$ is defined using functional calculus of unbounded operators. 
\end{lemma}
Proof can be taken from~\cite{GH04}.

Now we can state the map $\alpha$ we require to prove Bott periodicity.

\begin{proposition}
There exists a unique asymptotic morphism (up to equivalence)
\[\alpha_t \colon \mathcal{S} \widehat{\otimes}  C_0(V,\text{Cliff}(V)) \dashrightarrow \mathcal{K}(\mathcal{H}(V))\]
defined by 
\[\alpha_t(f \widehat{\otimes} h) = f(t^{-1}D) M_{h_t},\]
for $t \geq 1$.
\end{proposition}
\begin{proof}
Firstly to prove this we consider $\alpha_t$ defined from the algebraic tensor product of $\mathcal{S} \widehat{\odot} C_0(V,\text{Cliff}(V)) $ to $\mathcal{K}(\mathcal{H}(V))$ by, 
\[\alpha_t(f \widehat{\odot} h) = f(t^{-1}D) M_{h_t}.\]
Then we have a linear map, which is asymptotically $\ast$-linear by Lemma~\ref{doublecommutatordisappears} and universality of the tensor product gives an asymptotic morphism $$\alpha_t(f \widehat{\otimes} h) = f(t^{-1}D) M_{h_t}.$$ Also by Lemma~\ref{compactoperator}, the image is contained in the compact operators. Hence we have our required asymptotic  morphism. 
\end{proof}

\begin{definition}
On a finite dimensional Euclidean real vector space we define the \emph{Clifford operator} $C\colon \text{Sch}(V) \rightarrow \mathcal{H}(V)$ by 
\[(Cf)(v) = \sum_{i=1}^n x_i e_i(f(v)),\]
where $e_i$ form an orthornormal basis for $V$, $x_i$ are the corresponding coordinates in $V$. Here we use functional calculus to define the functions throughout the summation. 
\end{definition}

\begin{lemma}
The composition of the Bott map and the multiplication operator given by 
\[ \mathcal{S} \xrightarrow{\beta} C_0(V, \text{Cliff}(V)) \xrightarrow{M} \mathcal{B}(\mathcal{H}(V)) ,\]
is defined by 
\[f \mapsto f(C),\]
where $C$ is the Clifford operator and $f \in \mathcal{S}$. 
\end{lemma}
\begin{proof}
Since the operator $C$ is essentially self adjoint on the Hilbert space $\mathcal{H}(V)$ we can form the operator $f(C)$ in the bounded linear operators. Then 
\[(M \beta(f))(g)(v) = M_{\beta(f)} (g) (v) = \beta(f)(v) g(v) = f(C)(v)g(v) = f(C)g(v).\] 
\end{proof}
Now we need to prove that $\alpha$ and $\beta$ satisfy the requirements that their induced maps give $\alpha_{\ast}\beta_{\ast}$ is homotopic to $1 \in K(\mathbb{R})$. 
In order to show this we firstly need an equivalent morphism coming from composing $\alpha$ and $\beta$ for which we need a new operator. 

\begin{definition}
We define an operator $B \colon \text{Sch}(V) \rightarrow \text{Sch}(V)$  by 
\[(Bf)(v) = \sum_{i=1}^{n} x_i e_i(f(v)) + \sum_{i=1}^{n} \hat{e_i} (\frac{\partial f}{\partial x_i}(v)).\] 
Call $B$ the \emph{Bott operator} and notice $B = C+D$.
\end{definition}

\begin{definition}
Define the \emph{number operator} $N\colon  \text{Cliff}(V) \rightarrow \text{Cliff}(V)$ by 
\[N = \sum_{i=1}^n \widehat{e_i} e_i,\]
where $e_i$'s form an orthonormal basis for $V$. 
\end{definition} 
We should note that the number operator is bounded and we can extend it to an operator on $\mathcal{H}(V)$ by composition with a map from $V$ to $\text{Cliff}(V)$. 
The following is an easy calculation. 
\begin{proposition}\label{B^2equals}
\[B^2 = C^2 + D^2 +N.\]
\end{proposition}

The following two results have identical proof to the results stated in~\cite{GH04} so for proofs see Proposition~1.16 and Corollary~1.1 of that paper.
\begin{proposition}\label{eigenvalues}
For a finite dimension real vector space $V$ of dimension $n$, consider the operator $B=C+D$. Then there exists an orthornormal basis in $\text{Sch}(V)$ for $\mathcal{H}(V)$ consisting of the eigenvectors for $B^2$ such that 
\begin{enumerate}
\item the eigenvalues are nonnegative integers and each eigenvalue occurs with finite multiplicity, 
\item the eigenvalue 0 occurs only once and its corresponding eigenfunction is $e^{-\frac{1}{2} ||v||^2}$.
\end{enumerate}
\end{proposition}

\begin{corollary}\label{rank1}
For a finite dimensional real vector space $V$. Then the Bott operator $B$ 
\begin{enumerate}
\item is essentially self-adjoint,
\item has compact resolvent,
\item has one-dimensional kernel generated by the function $e^{-\frac{1}{2} ||v||^2}$.
\end{enumerate}
\end{corollary}

The following follows from~\cite{HRT} page~114-115, and it is stated in~\cite{GH04} as Proposition~1.6. 
\begin{proposition}[Mehler's formula] \label{Mehler}
For the Clifford operator $C$ and the Dirac operator $D$ defined above we have the following identities for $s > 0$
\[e^{-s(C^2 +D^2)} = e^{-\frac{1}{2} s_1 C^2} e^{-s_2 D^2} e^{-\frac{1}{2} s_1 C^2} ~~ \text{and} ~~ e^{-s(C^2 +D^2)} = e^{-\frac{1}{2} s_1 D^2} e^{-s_2 C^2} e^{-\frac{1}{2} s_1 D^2}\]
with \[s_1 = \frac{\text{cosh}(2s) -1}{\text{sinh}(2s)} ~ \text{and} ~  s_2 = \frac{\text{sinh}(2s)}{2}.\]
\end{proposition}
\qed 

Then as in~\cite{GH04}, Lemma~1.11, we have the asymptotic conditions:

\begin{lemma}
For an unbounded operator $X$ we have the following conditions: 

\[\lim_{t \to \infty}||e^{-\frac{1}{2} s_1 X^2} - e^{-\frac{1}{2} t^{-2} X^2} || =0,\]
\[\lim_{t \to \infty}||e^{-\frac{1}{2} s_2 X^2} - e^{-\frac{1}{2} t^{-2} X^2} || =0,\]
and 
\[\lim_{t \to \infty}||t^{-1} Xe^{-\frac{1}{2} s_1 X^2} - t^{-1}Xe^{-\frac{1}{2} t^{-2} X^2} || =0,\]
\[\lim_{t \to \infty}||t^{-1} Xe^{-\frac{1}{2} s_2 X^2} - t^{-1}Xe^{-\frac{1}{2} t^{-2} X^2} || =0,\]
where  \[s_1 = \frac{\text{cosh}(2t^{-2}) -1}{\text{sinh}(2t^{-2})} ~ \text{and} ~  s_2 = \frac{\text{sinh}(2t^{-2})}{2}.\]
\end{lemma}
\qed

For a proof of the following result, see Lemma~1.12 of~\cite{GH04}.
\begin{lemma}\label{commutativegrading} 
Let $f,g \in C_0(\mathbb{R})$, then 
\[\lim_{t \to \infty} || [f(t^{-1}C), g(t^{-1}D)] || = 0. \]
\end{lemma}
\qed

We can use the proof in~\cite{GH04} of Theorem~1.17 and the above statements to obtain.
\begin{proposition}\label{asyequivofCandDwithu}
For operators $C$ and $D$ defined before, 
\[ e^{-t^{-2} B^2} \sim_{asy} e^{-t^{-2} C^2} e^{-t^{-2} D^2}.\]
\end{proposition}
\qed

The following is similar.

\begin{proposition}\label{asyequivofCandDwithv}
For the operators $C$ and $D$, 
\[ t^{-1} B e^{-t^{-2} B^2} \sim_{asy} t^{-1}(C+D) e^{-t^{-2} C^2} e^{-t^{-2} D^2}.\]
\end{proposition}
\qed

\begin{thm}\label{composition equivalent to gamma}
The composition of $\alpha_t$ and $\beta$ given by :
\[ \xymatrixcolsep{3pc}\xymatrix{\mathcal{S} \ar[r]^-{\Delta} & \mathcal{S} \widehat{\otimes} \mathcal{S} \ar[r]^-{\text{id} \widehat{\otimes} \beta} & \mathcal{S} \widehat{\otimes} C_0(V,\text{Cliff}(V)) \ar@{.>}[r]^-{\alpha_t} & \mathcal{K}(\mathcal{H}(V),} \]
is asymptotically equivalent to the asymptotic morphism $\gamma_t \colon \mathcal{S} \dashrightarrow \mathcal{K}(H)$ defined by
\[\gamma_t(f) = f(t^{-1} B),\]
for all $t \geq 1$.
\end{thm}
\begin{proof}
Here we imitate the proof of Guentner and Higson in~\cite{GH04} with additional details.
Since $S$ is generated by  $u$ and $v$ by Lemma~\ref{Agenerates} it suffices to check that 
\[\alpha(\text{id} \widehat{\otimes} \beta)(\Delta(f)) \sim_{asy} \gamma_t(f),\]
for $f=u$ and $f=v$. 

For $f=u$, 
\begin{align*}
\alpha(\text{id} \widehat{\otimes} \beta)(\Delta(u)) 
& = \alpha(\text{id} \widehat{\otimes} \beta)(u \widehat{\otimes} u) \\
& = \alpha(u \widehat{\otimes} \beta(u)) \\
& = \alpha (u \widehat{\otimes} u(C)) \\
& = u(t^{-1}D) M_{u(C)_t} \\
& =  u(t^{-1}D) u(t^{-1}C) \\
& = e^{-t^{-2} C^2} e^{-t^{-2} D^2},
\end{align*}
and \[\gamma_t(u) = u(t^{-1} B) =  e^{-t^{-2} B^2},\]
and these are both asymptotically equivalent by Proposition~\ref{asyequivofCandDwithu}.

For $f=v$,
\begin{align*}
\alpha(\text{id} \widehat{\otimes} \beta)(\Delta(u)) 
& = \alpha(\text{id} \widehat{\otimes} \beta)(u\widehat{\otimes} v +  v \widehat{\otimes} u) \\
& = \alpha(u\widehat{\otimes} \beta(v) +  v \widehat{\otimes} \beta(u)) \\
& =  \alpha(u\widehat{\otimes} v(C) +  v \widehat{\otimes} u(C)) \\
& \sim_{asy} \alpha(u\widehat{\otimes} v(C)) + \alpha( v \widehat{\otimes} u(C)) \\
& = u(t^{-1}D) M_{v(C)_t} +  v(t^{-1}D) M_{u(C)_t} \\
& =  u(t^{-1}D) v(t^{-1}C)+ v(t^{-1}D) u(t^{-1}C) \\
& = e^{-t^{-2} D^2} t^{-1}C e^{-t^{-2} C^2}+t^{-1} D e^{-t^{-2} D^2}e^{-t^{-2} C^2} \\
& = t^{-1}(C+D) e^{-t^{-2} C^2} e^{-t^{-2} D^2},
\end{align*}
and $$\gamma_t(v) = v(t^{-1}B) = t^{-1}B e^{-t^{-2} B^2}.$$ Then by Proposition~\ref{asyequivofCandDwithv} these are asymptotically equivalent. 
Hence the composition of $\alpha$ and $\beta$ is asymptotically equivalent to $\gamma_t$.  
\end{proof}

\begin{corollary}\label{k-theorygroupsareisomorphic}
The composition $\alpha_{\ast} \beta_{\ast}$ of the induced homomorphisms $$\beta_\ast \colon K(\mathcal{K}(\mathcal{H})) \rightarrow K(C_0(V, \text{Cliff}(V)))$$ and $$\alpha_{\ast} \colon K(C_0(V,\text{Cliff}(V))) \rightarrow K(\mathbb{\mathcal{K}(\mathcal{H})}),$$ is the identity homomorphism.
\end{corollary}
\begin{proof}
By Theorem~\ref{composition equivalent to gamma}, the composition of $\alpha_{\ast} \beta_{\ast}$ is equivalent to the asymptotic morphism $\gamma_t$, given by $\gamma_t(f) = f(t^{-1}B)$. Since $f$ is in $\mathcal{S}$, each $\gamma_t$ is a $\ast$-homomorphism and $\gamma$ is homotopic to the $\ast$-homomorphism mapping $f$ to $f(B)$. 
Now it suffices to define a homotopy between $\gamma$ and the map $\theta \colon \mathcal{S} \rightarrow \mathcal{K}(\mathcal{H})$ defined by $\theta(f) = f(0)p$ where $p$ is a rank~1 projection (by Corollary~\ref{rank1}) onto the kernel of $B$.

%

Let $f(x)=u(x)=e^{-x^2}$ or $f(x)=v(x)=x e^{-x^2}$.

Consider an eigenvector $w$ of $u(B)$ or $v(B)$ with non-zero eigenvalue, then 
\[u(t^{-1} B) w \to 0 ~ ~ \text{and}   ~ ~ v(t^{-1} B) w \to 0\]
as $t \to 0$. 

Now if $w \in \text{ker}(B)$, then
\[u(t^{-1} B) w = w ~ ~ \text{and}   ~ ~ v(t^{-1} B) w = 0\]
for all $t$.

Finally, for an eigenvector $w$ of $u(B)$ or $v(B)$ respectively 
\begin{equation*}
u(0)p(w)= 
\begin{cases}
0 & \text{if}\; w \notin \text{ker}(B) \\
w & \text{if} \; w \in \text{ker}(B),
\end{cases}
\end{equation*}
and $v(0)p(w) = 0$.

By Proposition~\ref{eigenvalues} and Corollary~\ref{rank1}, $\mathcal{H}(V)$ has a basis consisting of eigenvectors of $u(B)$ and $v(B)$.
Combining these, we have 
\[||f(t^{-1}B) - f(0)p|| \to 0,\]
as $t \to 0$ when $f=u$ or $f=v$. Since $u$ and $v$ generate $\mathcal{S}$, it follows that the above holds for all $f \in \mathcal{S}$, and so we have a homotopy between $\gamma$ and $\theta$ defined by
\begin{equation*}
f \mapsto
\begin{cases}
f(s^{-1}B) & \text{if}\; s \in (0, 1] \\
f(0)p& \text{if} \; s =0,
\end{cases}
\end{equation*}
where $p$ is the projection onto the kernel of $B$. 
\end{proof}

The converse of this now follows below. 

\begin{definition}
Define the \emph{flip map} 
\[ l \colon A \widehat{\otimes} B \rightarrow B \widehat{\otimes} A,\]
by 
\[ l(a \widehat{\otimes} b) = (-1)^{\text{deg}(a) \text{deg}(b)} b \widehat{\otimes} a,\]
for all $a \in A$ and $b \in B$. 
\end{definition}
Notice that $l$ is a $\ast$-isomorphism. Let $\mathcal{C}(V) = C_0(V,\text{Cliff}(V))$ for simplicity throughout the following statements.

\begin{lemma} \label{asymptoticallycommutes}
The diagram 
\[\xymatrixcolsep{3pc}\xymatrixrowsep{3pc}\xymatrix{
\mathcal{S} \widehat{\otimes} \mathcal{C}(V) \ar[d]^{\Delta \widehat{\otimes} \text{id}_{\mathcal{C}(V)}}  \ar[r]^{\Delta \widehat{\otimes} \text{id}_{\mathcal{C}(V)}}& \mathcal{S} \widehat{\otimes} \mathcal{S}\widehat{\otimes} \mathcal{C}(V) \ar@{.>}[r]^{\text{id}_{\mathcal{S}} \widehat{\otimes} \alpha_t} &\mathcal{S} \widehat{\otimes} \mathcal{K}(\mathcal{H}) \ar[r]^{ \beta \widehat{\otimes}\text{id}_{\mathcal{K}(\mathcal{H})}} & \mathcal{C}(V) \widehat{\otimes} \mathcal{K}(\mathcal{H}) \ar[d]_{\l} \\ 
\mathcal{S} \widehat{\otimes} \mathcal{S}\widehat{\otimes} \mathcal{C}(V) \ar[r]_{\text{id}_{\mathcal{S}} \widehat{\otimes} \beta \widehat{\otimes}\text{id}_{\mathcal{C}(V)}}  & \mathcal{S} \widehat{\otimes} \mathcal{C}(V) \widehat{\otimes} \mathcal{C}(V) \ar[r]_{\text{id}_{\mathcal{S}} \widehat{\otimes} l}  & \mathcal{S} \widehat{\otimes} \mathcal{C}(V)\widehat{\otimes} \mathcal{C}(V)  \ar@{.>}[r]_{\alpha_t \widehat{\otimes} \text{id}_{\mathcal{C}(V)}} & \mathcal{K}(\mathcal{H}) \widehat{\otimes} \mathcal{C}(V),} \]
asymptotically commutes.
\end{lemma}
\begin{proof}
Since $\mathcal{S}$ is generated by the elements $u(x) = e^{-x^2}$ and $v(x) =x e^{-x^2}$ it suffices to check that the diagram asymptotically commutes for $u$ and $v$ in $\mathcal{S}$. 
Let $$f_1 =  l(\beta \widehat{\otimes}\text{id}_{\mathcal{K}(\mathcal{H})})(\text{id}_{\mathcal{S}} \widehat{\otimes} \alpha_t) (\Delta \widehat{\otimes} \text{id}_{\mathcal{C}(V)})  $$
and $$g_1 = (\alpha_t \widehat{\otimes} \text{id}_{\mathcal{C}(V)}) (\text{id}_{\mathcal{S}} \widehat{\otimes} l) (\text{id}_{\mathcal{S}} \widehat{\otimes} \beta \widehat{\otimes}\text{id}_{\mathcal{C}(V)})(\Delta \widehat{\otimes} \text{id}_{\mathcal{C}(V)})$$
For $u$ and $h \in \mathcal{C}(V)$, 
\begin{align*}
f_1 (u \widehat{\otimes} h) 
&=  l(\beta \widehat{\otimes}\text{id}_{\mathcal{K}(\mathcal{H})})(\text{id}_{\mathcal{S}} \widehat{\otimes} \alpha_t)  (u \widehat{\otimes} u \widehat{\otimes} h) \\
& =   l(\beta \widehat{\otimes}\text{id}_{\mathcal{K}(\mathcal{H})}) (u \widehat{\otimes} u(t^{-1} D) M_{h_t}) \\
& =l  (u(C) \widehat{\otimes} u(t^{-1} D) M_{h_t})\\
& =  u(t^{-1} D) M_{h_t}\widehat{\otimes}u(C),
\end{align*}
and  
\begin{align*}
g_1(u \widehat{\otimes} h)
& =  (\alpha_t \widehat{\otimes} \text{id}_{\mathcal{C}(V)}) (\text{id}_{\mathcal{S}} \widehat{\otimes} l) (\text{id}_{\mathcal{S}} \widehat{\otimes} \beta \widehat{\otimes}\text{id}_{\mathcal{C}(V)})(u \widehat{\otimes}u \widehat{\otimes}  h)\\ 
& = (\alpha_t \widehat{\otimes} \text{id}_{\mathcal{C}(V)}) (\text{id}_{\mathcal{S}} \widehat{\otimes} l)  (u \widehat{\otimes}u(C) \widehat{\otimes}  h)\\ 
& =    (\alpha_t \widehat{\otimes} \text{id}_{\mathcal{C}(V)})(u \widehat{\otimes}  h\widehat{\otimes}u(C))\\ 
& = u(t^{-1} D) M_{h_t} \widehat{\otimes} u(C).
\end{align*}
Hence we have an asymptotic equivalence in the case when we take $u$. 
Now consider $v$. 
We have, 
\begin{align*}
f_1 (v \widehat{\otimes} h) 
&= l(\beta \widehat{\otimes}\text{id}_{\mathcal{K}(\mathcal{H})})(\text{id}_{\mathcal{S}} \widehat{\otimes} \alpha_t) ((u \widehat{\otimes} v +v \widehat{\otimes} u) \widehat{\otimes} h) \\
& \sim_{\text{asy}}  l(\beta \widehat{\otimes}\text{id}_{\mathcal{K}(\mathcal{H})})(u \widehat{\otimes} v(t^{-1} D)M_{h_t} +v \widehat{\otimes} u(t^{-1} D)M_{h_t}) \\
& =l (u(C) \widehat{\otimes} v(t^{-1} D)M_{h_t} +v(C)  \widehat{\otimes} u(t^{-1} D)M_{h_t}) \\ 
& = v(t^{-1} D)M_{h_t}\widehat{\otimes} u(C) + u(t^{-1} D)M_{h_t}  \widehat{\otimes} v(C),
\end{align*}
and 
\begin{align*}
g_1(v \widehat{\otimes} h) 
& =(\alpha_t \widehat{\otimes} \text{id}_{\mathcal{C}(V)}) (\text{id}_{\mathcal{S}} \widehat{\otimes} l) (\text{id}_{\mathcal{S}} \widehat{\otimes} \beta \widehat{\otimes}\text{id}_{\mathcal{C}(V)})((u \widehat{\otimes} v +v \widehat{\otimes} u) \widehat{\otimes} h) \\
& =  (\alpha_t \widehat{\otimes} \text{id}_{\mathcal{C}(V)}) (\text{id}_{\mathcal{S}} \widehat{\otimes} l) (u \widehat{\otimes} v(C)\widehat{\otimes} h +v \widehat{\otimes} u(C)\widehat{\otimes} h) \\
& =   (\alpha_t \widehat{\otimes} \text{id}_{\mathcal{C}(V)}) (u \widehat{\otimes} h \widehat{\otimes} v(C)+v \widehat{\otimes} h\widehat{\otimes} u(C)) \\
& \sim_{\text{asy}}  u(t^{-1}D)M_{h_t} \widehat{\otimes} v(C)+v(t^{-1}D)M_{h_t} \widehat{\otimes} u(C) \\
\end{align*}
and again we have the diagram commuting asymptotically. 
\end{proof}

\begin{lemma}\label{flipmaphomotopic}
The flip map on $\mathcal{C}(V) \widehat{\otimes} \mathcal{C}(V)$ is homotopic via graded $\ast$-homomorphisms to the map $h_1 \widehat{\otimes} h_2 \mapsto h_1 \widehat{\otimes} \iota h_2$, where $\iota$ is the automorphism on $\mathcal{C}(V)$ induced by the map $e \mapsto -e$ in $V$. 
\end{lemma}
A proof can be found in \cite{HKT98}, Lemma~18. 

\begin{corollary}\label{isofrombeta}
There is an isomorphism: 
$$K(\mathbb{R}) \rightarrow K(\mathcal{C}(V))$$ 
induced by $\beta\colon  S \rightarrow \mathcal{C}(V)$. 
\end{corollary}
\begin{proof}
It suffices to check that the $f_1$ in the proof of Lemma~\ref{asymptoticallycommutes} induces an isomorphism in $K$-theory since by Corollary~\ref{k-theorygroupsareisomorphic} we have that $\alpha_\ast$ is left inverse to $\beta$. It follows from the stability of $K$-theory that:
\[K(\mathcal{S} \widehat{\otimes} \mathcal{C}(V)) \xrightarrow{\alpha_\ast} K(\mathcal{K}(\mathcal{H})) \cong K(\mathbb{R}).\]
So we need to show: 
\[\xymatrixcolsep{3pc}\xymatrixrowsep{3pc}\xymatrix{\mathcal{S} \widehat{\otimes} \mathcal{C}(V) \ar[r]^{\Delta \widehat{\otimes} \text{id}_{\mathcal{C}(V)}}& \mathcal{S} \widehat{\otimes} \mathcal{S}\widehat{\otimes} \mathcal{C}(V) \ar@{.>}[r]^{\text{id}_{\mathcal{S}} \widehat{\otimes} \alpha_t}& \mathcal{S} \widehat{\otimes} \mathcal{K}(\mathcal{H})\ar[r]^{ \beta \widehat{\otimes}\text{id}_{\mathcal{K}(\mathcal{H})}} & \mathcal{C}(V) \widehat{\otimes} \mathcal{K}(\mathcal{H})  \ar[r]^{l} & \mathcal{K}(\mathcal{H}) \widehat{\otimes} \mathcal{C}(V) }
\]
induces an isomorphism on $K$-theory. 
By Lemma~\ref{asymptoticallycommutes}, we have an asymptotic equivalence which gives a homotopy equivalence. Further the composition in the diagram above is asymptotically equivalent to $$g_1 = (\alpha_t \widehat{\otimes} \text{id}_{\mathcal{C}(V)}) (\text{id}_{\mathcal{S}} \widehat{\otimes} l) (\text{id}_{\mathcal{S}} \widehat{\otimes} \beta \widehat{\otimes}\text{id}_{\mathcal{C}(V)})(\Delta \widehat{\otimes} \text{id}_{\mathcal{C}(V)})$$ which is homotopic by Lemma~\ref{flipmaphomotopic} to the composition 
$$(\alpha_t(\text{id}_{\mathcal{S}} \widehat{\otimes} \beta)\Delta) \widehat{\otimes} \iota$$
and hence by Theorem~\ref{composition equivalent to gamma} is homotopic to 
$\gamma \widehat{\otimes} \iota$, and so maps
$$ f \widehat{\otimes} h \mapsto f(t^{-1} B) \widehat{\otimes} \iota(h).$$
and so induces an isomorphism in $K$-theory by the homotopy defined similarly to that in the proof of Corollary~\ref{k-theorygroupsareisomorphic}
\begin{equation*}
f \widehat{\otimes} h \mapsto
\begin{cases}
 f(s^{-1}B) \widehat{\otimes} \iota(h) & \text{if}\; s \in (0, 1] \\
f(0) p \widehat{\otimes} \iota(h) & \text{if} \; s =0.
\end{cases}
\end{equation*}
\end{proof}
Then the proof of Theorem~\ref{Bott periodicty statement in beta} is complete, since we have shown that the map
\[K(\mathbb{R}) \rightarrow K(\mathbb{R} \widehat{\otimes} C_0(V, \text{Cliff}(V))),\]
is an isomorphism.

Observe that when $V=\mathbb{R}$ then 
\[C_0(V,\text{Cliff}(V)) = C_0(\mathbb{R},\text{Cliff}(\mathbb{R})) \cong C_0(\mathbb{R}) \widehat{\otimes} \text{Cliff}(\mathbb{R})) \cong \Sigma \widehat{\otimes} \mathbb{R}_{1,0}.\]
Then we obtain

\begin{corollary}
\[K(\mathbb{R}) \cong K(\Sigma \mathbb{R} \widehat{\otimes} \mathbb{R}_{1,0}).\]
\end{corollary}
\qed

\begin{corollary}
For any real graded $C^\ast$-algebra, 
\[K(A) \cong K(\Sigma A \widehat{\otimes} \mathbb{R}_{1,0}).\]
\end{corollary}
The proof of this follows from the $K$-theory product described in section~1.7 of~\cite{GH04}.

Then combining this with Theorem~1.14 in~\cite{GH04}, we have  

\begin{thm}
For any graded $C^\ast$-algebra over $\mathbb{F} = \mathbb{R}~\text{or}~ \mathbb{C}$, 
\[K_{\mathbb{F}} (A) \cong K_{\mathbb{F}}(\Sigma A \widehat{\otimes} \mathbb{F}_{1,0}) ,\]
where $K_{\mathbb{F}}$ denotes real or complex $K$-theory.
\end{thm}

We finally obtain the isomorphism in $E$-theory. 
\begin{thm}
For any graded $C^\ast$-algebra over $\mathbb{F} = \mathbb{R}~\text{or}~ \mathbb{C}$, 
\[E_{\mathbb{F}} (A,B) \cong E_{\mathbb{F}}(A,\Sigma B \widehat{\otimes} \mathbb{F}_{1,0}) ,\]
and 
\[E_{\mathbb{F}} (A,B) \cong E_{\mathbb{F}}(\Sigma A \widehat{\otimes} \mathbb{F}_{1,0},B) ,\]
where $E_{\mathbb{F}}$ denotes real or complex $E$-theory.
\end{thm}
\begin{proof}
The Bott map $\beta: \mathcal{S} \rightarrow \Sigma \widehat{\otimes} \mathbb{F}_{1,0}$ gives an invertible class $[\beta] \in E_{\mathbb{F}}(\mathbb{F}, \Sigma \mathbb{F} \widehat{\otimes} \mathbb{F}_{1,0})$. Also we have an invertible class 
\[ [\beta \widehat{\otimes} \text{id}_A] \in E_{\mathbb{F}}(\mathbb{F} \widehat{\otimes} A, \Sigma \mathbb{F} \widehat{\otimes} \mathbb{F}_{1,0}\widehat{\otimes} A) = E_{\mathbb{F}}( A, \Sigma A \widehat{\otimes} \mathbb{F}_{1,0}).\]
Denote this class by $[\beta_A]$ and similarly set $[\beta_B]$ for the class $[\beta \widehat{\otimes} \text{id}_B]$.
Then these classes are invertible. 

By the $E$-theory product we have a group homomorphism 
\[ f \colon E_{\mathbb{F}} (A,B) \rightarrow E_{\mathbb{F}}(A,\Sigma B \widehat{\otimes} \mathbb{F}_{1,0})\]
defined by 
\[f(x) = x \times [\beta_B]\]
for all $x \in E_{\mathbb{F}}(A,B)$ which is invertible with inverse defined by 
\[f^{-1}(y) = y \times [\beta_B]^{-1}\]
for all $y \in E_{\mathbb{F}}(A,\Sigma B \widehat{\otimes} \mathbb{F}_{1,0})$.  

Similarly we define 
\[g \colon E_{\mathbb{F}} (A,B) \rightarrow E_{\mathbb{F}}(\Sigma A \widehat{\otimes} \mathbb{F}_{1,0},B)\]
by 
\[g(x) = [\beta_A]^{-1} \times x\]
for all $x \in E_{\mathbb{F}} (A,B)$, and this has inverse defined by 
$$g^{-1}(y) = [\beta_A] \times y$$, for all $y \in E_{\mathbb{F}}(\Sigma A \widehat{\otimes} \mathbb{F}_{1,0},B)$.
Then the result follows. 

%
%
\end{proof}

Now recalling some certain facts of Clifford algebras we will see we can obtain the $8$-fold periodicity. We have that $\mathbb{R}_{p,q} \widehat{\otimes} \mathbb{R}_{r,s} \cong \mathbb{R}_{p+r,q+s}$ and also $\mathbb{R}_{8,0} \cong \mathbb{R}_{4,4}$, and also the result from the PhD thesis of the author
\[E^n(A,B) \cong E^n(A \widehat{\otimes} \mathbb{R}_{1,1}, B \widehat{\otimes} \mathbb{R}_{1,1}),\]
then we obtain using the Bott periodicty result the $8$-fold periodicity of $E$-theory,

\begin{corollary}
Let $A$ and $B$ be real graded $C^\ast$-algebra. Then we have natural isomorphisms
\[E_g^n(A,B) \cong E_g^{n+8}(A,B).\]
\end{corollary}
\begin{proof}
\begin{align*}
E_g^n(A,B) 
& =  E_g^{n}(A,\Sigma^8 (B \widehat{\otimes} \mathbb{R}_{1,0} \widehat{\otimes} \mathbb{R}_{1,0} \widehat{\otimes} \mathbb{R}_{1,0} \widehat{\otimes} \mathbb{R}_{1,0} \widehat{\otimes} \mathbb{R}_{1,0} \widehat{\otimes} \mathbb{R}_{1,0} \widehat{\otimes} \mathbb{R}_{1,0} \widehat{\otimes} \mathbb{R}_{1,0})), \\
& = E_g^{n}(A,\Sigma^8 (B \widehat{\otimes} \mathbb{R}_{4,4})), \\
& = E_g^{n}(A,\Sigma^8 B), \\
& \cong E_g^{n+8}(A, B). \\ 
\end{align*}
\end{proof}

\end{document}